\documentclass[a4paper]{article}
\usepackage[english]{babel}
\usepackage[utf8x]{inputenc}
\usepackage[T1]{fontenc}
\usepackage{amsthm} 
\usepackage[a4paper,top=3cm,bottom=2cm,left=3cm,right=3cm,marginparwidth=1.75cm]{geometry}
\usepackage{float}
\usepackage{amsmath}
\usepackage{graphicx}
\usepackage[colorinlistoftodos]{todonotes}
\usepackage[colorlinks=true, allcolors=blue]{hyperref}

\newtheorem{theorem}{Theorem}
\newtheorem*{theorem*}{Theorem}
\newtheorem{lemma}{Lemma}
\newtheorem{corollary}{Corollary}

\allowdisplaybreaks
\title{Topologically slice $(1,1)$-knots which are not smoothly slice}
\author{Zipei Nie}

\begin{document}
\maketitle

\begin{abstract}
We prove that there are infinitely many $(1,1)$-knots which are topologically slice, but not smoothly slice, which was a conjecture proposed by Béla András Rácz.
\end{abstract} 
\section{Introduction}
As defined in \cite{doll}, a knot $K$ is called a $(g,b)$-knot in $S^3$, if there is a Heegaard splitting $S^3=U \cup V$ of genus $g$, such that each of $K\cap U$ and $K \cap V$ consists of $b$ trivial arcs. By definition, the $(0,b)$-knots are the $b$-bridge knots. So the $(1,1)$-knots can be seen as $1$-bridge knots on the standard torus. All $2$-bridge knots and torus knots are $(1,1)$-knots.

The smooth (resp., topological) slice genus of a knot $K$ is the minimal genus of a connected, orientable  $2$-manifold smoothly (resp., locally flatly) embedded in the $4$-ball $D^4$ whose boundary is the knot $K$. A knot $K$ is called smoothly (resp.,topologically) slice, if its smooth (resp., topological) slice genus is zero. While it is known that \cite{gom} there exist infinitely many topologically slice knots which are not smoothly slice, we prove that it is still true if we restrict the knots to be $(1,1)$-knots, which was a conjecture proposed \cite{rac} by Rácz. In other words, our main theorem is the following. 

\begin{theorem*}
There are infinitely many $(1,1)$-knots which are topologically slice, but not smoothly slice.
\end{theorem*}

To prove the main theorem, we will construct a one-parameter family of $(1,1)$-knots $K_n$ ($n=0,1,\ldots$) as a valid example.  

For a $(1,1)$-knot $K$, the intersection of $K$ and the standard torus consists two basepoints $w$ and $z$. The information we need to determines the knot is the trivial arcs connecting $w$ and $z$ inside and outside the standard torus. However, knowing how to embed the two trivial arcs on the standard torus is more than enough. The meridian disks of the $1$-handles inside and outside the standard torus which does not intersect $K$ determines the knot $K$. The boundary of the meridian disks are called the $\alpha$ curve and the $\beta$ curve. A torus with a pair of basepoints $w$ and $z$ and a pair of curves $\alpha$ and $\beta$ on it, is called a $(1,1)$-diagram, if $\alpha$ and $\beta$ are embedded closed curves on the complement of $\{w,z\}$ and have the algebraic intersection number $\pm 1$. We can use four parameters to specify a $(1,1)$-diagram or a $(1,1)$-knot, which is called the Rasmussen's notation $K(p,q,r,s)$, as in \cite{ras}. 

In Section \ref{2}, we construct the family of knots and prove they are indeed $(1,1)$-knots by deriving the $(1,1)$-diagrams for them.

Given a $(1,1)$-knot $K=K(p,q,r,s)$, we can find \cite{goda} the boundary operator of the chain complex $CFK^\infty(S^3,K)$. Moreover, the invariant $\tau(K)$ defined in \cite{ozs} can be found because it only relies on $CFK^\infty(S^3,K)$. This invariant is closely related to the smooth slice genus $g_4(K)$ by the inequality $|\tau(K)|\le g_4(K)$, as demonstrated in \cite{ozs}. 

In Section \ref{3}, we compute the knot Floer homology of each $K_n$ in the above way. From that we prove these knots have trivial Conway polynomials, and therefore are topologically slice.

In Section \ref{4}, we first compute the $\tau$ invariant of $K_0$ in the demonstrated way. By certain inequalities for $\tau$ invariants in \cite{ozs}, we prove the $\tau$ invariant and the smooth slice genus of each $K_n$ are $1$, and therefore these knots are not smoothly slice.

The pictures of knots are generated by the software KnotPlot \cite{sch}.

The author would like to thank Zoltán Szabó for his help and support.

\section{The $(1,1)$-knots $K_n$}\label{2}
For each non-negative integer $n$, let $K_n$ be the knot with the following planar projection.

\begin{figure}[H]
\centering
\includegraphics[width=0.6\textwidth]{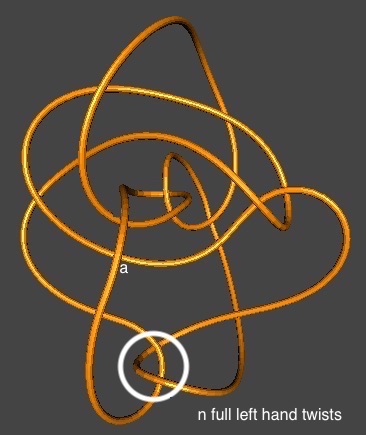}
\caption{\label{fig:k_n}Knot $K_1$ as shown. The knot $K_n$ has $n$ full left hand twists inside the white circle.}
\end{figure}

\begin{theorem}
$K_n$ is the $(1,1)$-knot $K(64n+31,24n+12,16n+6,32n+18)$ in Rasmussen's notation.
\end{theorem}
\begin{proof}
A $(1,1)$-diagram of $K_n$ is essentially a doubly-pointed Heegaard diagram $(\Sigma,\alpha,\beta, w,z)$, where $(\Sigma,\alpha,\beta)$ is a genus one Heegaard splitting of $S^3$ and $w,z$ are base points on the torus $\Sigma$. The following figure gives the torus $\Sigma$ as the blue torus and the base points $w,z$ as labeled. The $\alpha$ circle is immediately known. The only thing left to compute is the $\beta$ circle. This can be found via a sequence of isotopies of the space.

\begin{figure}[H]
\centering
\includegraphics[width=0.6\textwidth]{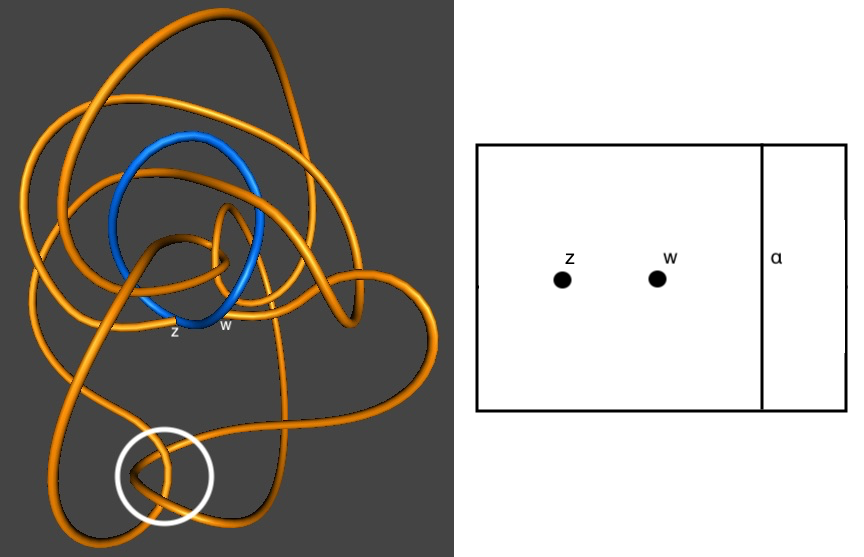}
\caption{First Step.}\label{fig:step1}
\end{figure}

Then, we move $z$ along a longitude of the blue torus clockwise and move $z$ around $w$ clockwise. Now we get the following figure.

\begin{figure}[H]
\centering
\includegraphics[width=0.57\textwidth]{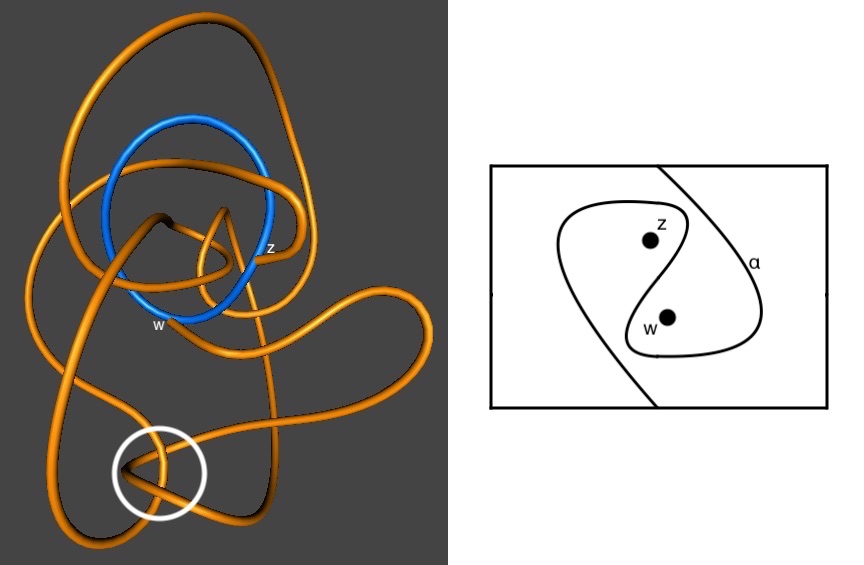}
\caption{Second Step.}\label{fig:step2}
\end{figure}

Then, we move $z$ along a longitude of the blue torus counterclockwise, and then move $z$ along a meridian of the blue torus into the paper. Now we get the following figure.

\begin{figure}[H]
\centering
\includegraphics[width=0.55\textwidth]{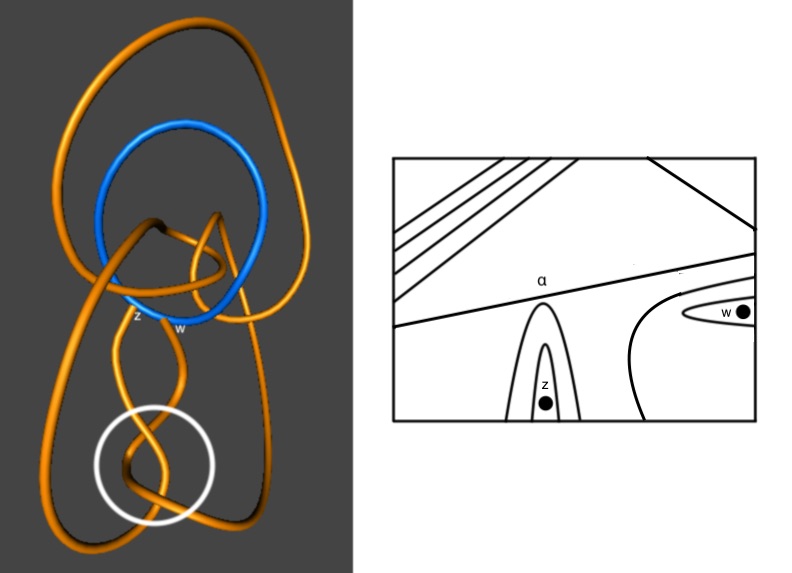}
\caption{Third Step.}\label{fig:step3}
\end{figure}

Now we move $w$ around $z$ clockwise $n$ times, and we get the following figure. 

\begin{figure}[H]
\centering
\includegraphics[width=0.25\textwidth]{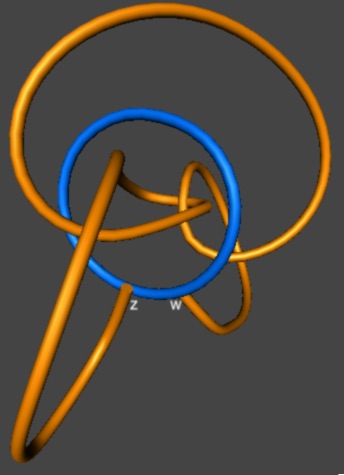}
\caption{Fourth Step.}\label{fig:step4}
\end{figure}

Then, we move $w$ along a longitude of the blue torus clockwise and and then move $w$ along a meridian of the blue torus out of the paper. Then we get the following figure. Here an arc with a red number $s$ on it represents a family of $s$ parallel curves.

\begin{figure}[H]
\centering
\includegraphics[width=0.55\textwidth]{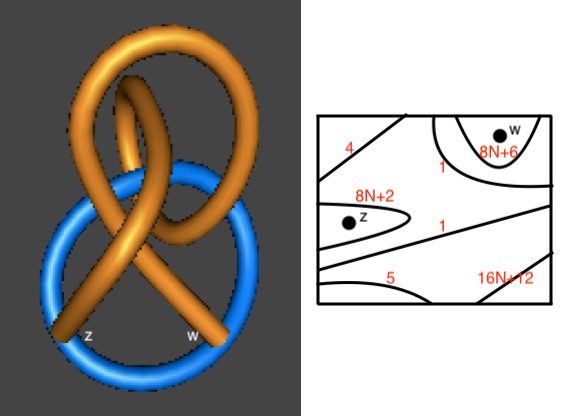}
\caption{Fifth Step.}\label{fig:step5}
\end{figure}

Before we proceed to find the $\beta$ curve, we simplify the $\alpha$ curve without moving the basepoints $w$ and $z$ and get the following diagram. 

\begin{figure}[H]
\centering
\includegraphics[width=0.55\textwidth]{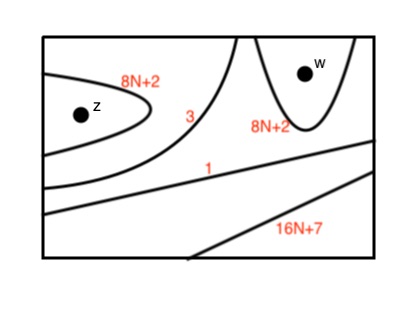}
\caption{Sixth Step.}\label{fig:step6}
\end{figure}

Finally, it is straightforward to get the triviality of the outside arc, and a $\beta$ curve can be constructed as follows (the blue curve). 

\begin{figure}[H]
\centering
\includegraphics[width=0.55\textwidth]{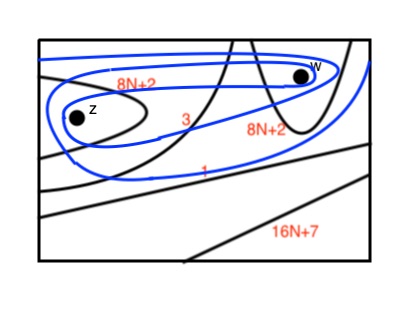}
\caption{Seventh Step.}\label{fig:step7}
\end{figure}

By straightening the $\beta$ curve, we obtain the following $(1,1)$-diagram.

\begin{figure}[H]
\centering
\includegraphics[width=0.55\textwidth]{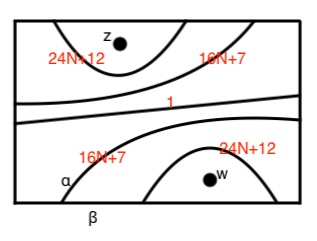}
\caption{Eighth Step.}\label{fig:step8}
\end{figure}

Equivalently, we have the following $(1,1)$-diagram. In Rasmussen's notation, $K_n$ is the $(1,1)$-knot $K(64n+31,24n+12,16n+6,32n+18)$.

\begin{figure}[H]
\centering
\includegraphics[width=0.55\textwidth]{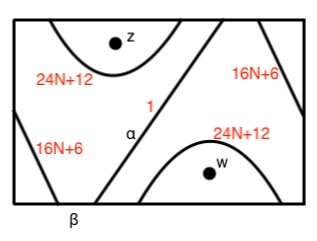}
\caption{Ninth Step.}\label{fig:step9}
\end{figure}

\end{proof}

\section{The knot Floer homology of $K_n$}\label{3}

\begin{theorem}
The Poincare polynomial of the knot Floer homology of $K_n$ is \begin{align*}
    \text{HF}_{K_n}(q,t)=&\sum_{m,a\in \mathbf{Z}} q^m t^a \text{rank}_{\mathbf{Z}}( \widehat{\text{HFK}}_m(S^3,K_n,a))\\=&-q^{-1}+(2n+1)q^{-3}t^{-2}(1+q)(1+q t)^4.
\end{align*}
\end{theorem}
\begin{proof}
It is clear that the total rank of the knot Floer homology of $K_n$ is $64n+31$, where the crossing points $x_i$ ($i=1,\ldots,64n+31$) between the $\alpha$ curve and the $\beta$ curve represent a basis of the knot Floer homology. To compute the knot Floer homology, we need to find the  Alexander grading $a_i$ and the Maslov grading $m_i$ of each crossing point $x_i$ ($i=1,\ldots,64n+31$).

From each Whitney disk of unit Maslov index, we derive a equation for the Alexander and Maslov grading of the vertices of the disk. With sufficient number of such disks, we can obtain the relative Alexander and Maslov gradings.

First, there is a disk from $x_{i}$ to $x_{48n+25-i}$ for each $i=1,\ldots,24n+12$, with a single basepoint $z$ inside. And there is a disk from $x_{i}$ to $x_{80n+39-i}$ 
for each $i=16n+8,\ldots,40n+19$, with a single basepoint $w$ inside. 

Then, we concentrate on the disk from $x_{24n+12}$ to $x_{24n+13}$, and try to extend it into larger disks while keeping the basepoint on it unchanged. Sequentially, we obtain disks from $x_{8(n-i)+2}$ to $x_{8(n-i)+1}$ and disks from $x_{8(n+i)+9}$ to $x_{8(n+i)+10}$ for $i=0,\ldots,n-1$. Then we obtain a disk from $x_{2}$ to $x_1$. After that, we sequentially obtain disks from $x_{8(2n-i)+5}$ to $x_{8(2n-i)+4}$ and disks from $x_{8i+6}$ to $x_{8i+7}$  for $i=0,\ldots,2n$. All the disks obtained here are Whitney disk of unit Maslov index with a single basepoint $z$ inside. 

Finally, we keep extending to obtain the disk from $x_{64n+31}$ to $x_{48n+25}$. To argue that the last one we got is an embedded disk in the universal covering space, we notice that, in every previous step, the segment where the endpoints lie is moving diagonally in the universal covering space, but in the last step, it is moving vertically, so the disk has no self intersections. Therefore the last disk we obtained is a Whitney disk of unit Maslov index with two $z$'s inside. 

Similarly, we extend the disk from $x_{40n+20}$ to $x_{40n+19}$ in the same way. Then we sequentially obtain disks from $x_{8(n+i)+5}$ to $x_{8(n+i)+6}$ and disks from $x_{8(n-i)-2}$ to $x_{8(n-i)-3}$ for $i=0,\ldots,n-1$. And then we obtain a disk from $x_{16n+5}$ to $x_{16n+6}$. Then we sequentially obtain disks from $x_{8i+2}$ to $x_{8i+3}$ and disks from $x_{8(2n-i)+1}$ to $x_{8(2n-i)}$  for $i=0,\ldots,2n-1$. And then we obtain a disk from $x_{16n+2}$ to $x_{16n+3}$. All the disks obtained here are Whitney disk of unit Maslov index with a single basepoint $w$ inside. And similarly, we obtain a Whitney disk of unit Maslov index with two $w$'s inside from $x_1$ to $x_{16n+7}$ in the final step.

Now we extend the disk from $x_{24n+10}$ to $x_{24n+15}$ in the same way. Then we sequentially obtain disks from $x_{8(n-i)+4}$ to $x_{8(n-i)-1}$ and disks from $x_{8(n+i)+7}$ to $x_{8(n+i)+12}$ for $i=0,\ldots,n-1$. All the disks obtained here are Whitney disk of unit Maslov index with a single basepoint $z$ inside. 

Similarly, we extend the disk from $x_{40n+22}$ to $x_{40n+17}$ in the same way. Then we sequentially obtain disks from $x_{8(n+i)+3}$ to $x_{8(n+i)+8}$ and disks from $x_{8(n-i)}$ to $x_{8(n-i)-5}$ for $i=0,\ldots,n-1$. Then we obtain a disk from $x_{16n+3}$ to $x_{16n+8}$. All the disks obtained here are Whitney disk of unit Maslov index with a single basepoint $w$ inside. 

Now we have found sufficient Whitney disks to get the relative Alexander and Maslov gradings for the first $16n+8$ crossing points. By symmetry, we also get the relative Alexander and Maslov gradings for the last $16n+8$ crossings. Via the Whitney disks from $x_i$ to $x_{80n+39-i}$ ($i=16n+8,\ldots 40n+19$), because we have a common element $x_{16n+8}$ in the sets of crossing points, we get the relative Alexander and Maslov gradings for the first $32n+15$ crossing points. By using the symmetry and the Whitney disks for a second time, we get the relative Alexander and Maslov gradings for all the crossing points.

The following is a solution to the relative Alexander gradings.
$$(a_{8i+1},a_{8i+2},a_{8i+3},a_{8i+4},a_{8i+5},a_{8i+6},a_{8i+7},a_{8i+8})=(-1,0,1,0,1,0,-1,0)$$
for $i=0,\ldots,n-1$.
$$a_{8n+1}=-1.$$
\begin{align*}
    &(a_{8n+8i+2},a_{8n+8i+3},a_{8n+8i+4},a_{8n+8i+5},a_{8n+8i+6},a_{8n+8i+7},a_{8n+8i+8},a_{8n+8i+9})\\
    =&(0,1,0,1,2,1,2,1)
\end{align*}
for $i=0,\ldots,2n$.
\begin{align*}
    &(a_{24n+8i+10},a_{24n+8i+11},a_{24n+8i+12},a_{24n+8i+13},a_{24n+8i+14},a_{24n+8i+15},a_{24n+8i+16},a_{24n+8i+17})\\=&(0,1,0,-1,0,-1,0,1)
\end{align*}
for $i=0,\ldots, n-1$.
\begin{align*}
    &(a_{32n+10},a_{32n+11},a_{32n+12},a_{32n+13},a_{32n+14},a_{32n+15},a_{32n+16},\\
    &a_{32n+17},a_{32n+18},a_{32n+19},a_{32n+20},a_{32n+21},a_{32n+22})\\
    =&(0,1,0,-1,0,-1,0,1,0,1,0,-1,0).
\end{align*}
\begin{align*}
    &(a_{32n+8i+23},a_{32n+8i+24},a_{32n+8i+25},a_{32n+8i+26},a_{32n+8i+27},a_{32n+8i+28},a_{32n+8i+29},a_{32n+8i+30})\\=&(-1,0,1,0,1,0,-1,0)
\end{align*}
for $i=0,\ldots,n-1$.
\begin{align*}
    &(a_{40n+8i+23},a_{40n+8i+24},a_{40n+8i+25},a_{40n+8i+26},a_{40n+8i+27},a_{40n+8i+28},a_{40n+8i+29},a_{40n+8i+30})\\
    =&(-1,-2,-1,-2,-1,0,-1,0)
\end{align*}
for $i=0,\ldots,2n$.
$$a_{56n+31}=1.$$
\begin{align*}
    &(a_{56n+8i+32},a_{56n+8i+33},a_{56n+8i+34},a_{56n+8i+35},a_{56n+8i+36},a_{56n+8i+37},a_{56n+8i+38},a_{56n+8i+39})\\
    =&(0,1,0,-1,0,-1,0,1)
\end{align*}
for $i=0,\ldots,n-1$.

By symmetry, the above is also a solution to the absolute Alexander gradings.

The following is a solution to the relative Maslov gradings.
$$(m_{8i+1},m_{8i+2},m_{8i+3},m_{8i+4},m_{8i+5},m_{8i+6},m_{8i+7},m_{8i+8})=(-2,-1,0,0,1,0,-1,-1)$$
for $i=0,\ldots,n-1$.
$$m_{8n+1}=-2.$$
\begin{align*}
    &(m_{8n+8i+2},m_{8n+8i+3},m_{8n+8i+4},m_{8n+8i+5},m_{8n+8i+6},m_{8n+8i+7},m_{8n+8i+8},m_{8n+8i+9})\\
    =&(-1,0,0,1,2,1,1,0)
\end{align*}
for $i=0,\ldots,n$.
\begin{align*}
    &(m_{16n+8i+10},m_{16n+8i+11},m_{16n+8i+12},m_{16n+8i+13},m_{16n+8i+14},m_{16n+8i+15},m_{16n+8i+16},m_{16n+8i+17})\\
    =&(-1,1,0,1,2,0,1,0)
\end{align*}
for $i=0,\ldots,n-1$.
\begin{align*}
    &(m_{24n+8i+10},m_{24n+8i+11},m_{24n+8i+12},m_{24n+8i+13},m_{24n+8i+14},m_{24n+8i+15},m_{24n+8i+16},m_{24n+8i+17})\\=&(-1,1,0,-1,0,-2,-1,0)
\end{align*}
for $i=0,\ldots, n-1$.
\begin{align*}
    &(m_{32n+10},m_{32n+11},m_{32n+12},m_{32n+13},m_{32n+14},m_{32n+15},m_{32n+16},\\
    &m_{32n+17},m_{32n+18},m_{32n+19},m_{32n+20},m_{32n+21},m_{32n+22})\\
    =&(-1,1,0,-1,0,-2,-1,0,0,1,0,-1,-1).
\end{align*}
\begin{align*}
    &(m_{32n+8i+23},m_{32n+8i+24},m_{32n+8i+25},m_{32n+8i+26},m_{32n+8i+27},m_{32n+8i+28},m_{32n+8i+29},m_{32n+8i+30})\\=&(-2,-1,0,0,1,0,-1,-1)
\end{align*}
for $i=0,\ldots,n-1$.
\begin{align*}
    &(m_{40n+8i+23},m_{40n+8i+24},m_{40n+8i+25},m_{40n+8i+26},m_{40n+8i+27},m_{40n+8i+28},m_{40n+8i+29},m_{40n+8i+30})\\
    =&(-2,-3,-2,-2,-1,0,-1,-1)
\end{align*}
for $i=0,\ldots,n-1$.
\begin{align*}
    &(m_{48n+8i+23},m_{48n+8i+24},m_{48n+8i+25},m_{48n+8i+26},m_{48n+8i+27},m_{48n+8i+28},m_{48n+8i+29},m_{48n+8i+30})\\
    =&(-2,-3,-1,-2,-1,0,-2,-1)
\end{align*}
for $i=0,\ldots,n$.
$$a_{56n+31}=0.$$
\begin{align*}
    &(m_{56n+8i+32},m_{56n+8i+33},m_{56n+8i+34},m_{56n+8i+35},m_{56n+8i+36},m_{56n+8i+37},m_{56n+8i+38},m_{56n+8i+39})\\
    =&(-1,1,0,-1,0,-2,-1,0)
\end{align*}
for $i=0,\ldots,n-1$.

To find the absolute Maslov gradings, we take a step back to look at Figure \ref{fig:step8}. If we remove basepoint $z$, most of the crossing points can be easily reduced. In fact, only the last three crossing points might survive. By analyzing the twisting part of the $\alpha$ curve, we find that the crossing points $x_{64n+29}$ and $x_{64n+30}$ can be reduced from above. Therefore, the only survived crossing point $x_{64n+31}$ has absolute Maslov grading zero. So the relative Maslov gradings we got is also the absolute Maslov gradings.

The Poincare polynomial is derived from counting the number of crossing points with given Alexander and Maslov grading.
\end{proof}

\begin{corollary}
$K_n$ and $K_m$ are non-isomorphic if $m\neq n$.
\end{corollary}
\begin{proof}
This is because they have different knot Floer homology.
\end{proof}
\begin{corollary}
The Conway polynomial of $K_n$ is $1$.
\end{corollary}
\begin{proof}
The Conway polynomial is the graded Euler characteristic of the knot Floer homology, so we have
$$\Delta_{K_n}(t)=\text{HF}_{K_n}(-1,t)=1.$$
\end{proof}
\begin{corollary}
$K_n$ is topologically slice.
\end{corollary}
\begin{proof}
Any knot with trivial Conway polynomial is topologically slice, as proved in \cite{fre,gar}.
\end{proof}

\section{The $\tau$ invariant and the smooth slice genus of $K_n$}\label{4}
\begin{lemma}
The invariant $\tau(K_0)$ is $1$.
\end{lemma}
\begin{proof}
To calculate the invariant $\tau(K_0)$, we first compute the chain complex $\text{CFK}^{\infty}(S^3,K_0)$. By definition, the generators of $\text{CFK}^{\infty}(S^3,K_0)$ is given by $[x_i,j,j+a_i]$ ($i=1,\ldots,64n+31, j\in \mathbf{Z}$), and the boundary operator on $\text{CFK}^{\infty}(S^3,K_0)$ is given by
\begin{align*}
\partial [x_1,i+1,i]&=[x_{7},i-1,i]-[x_{10},i,i]+[x_{21},i,i-1]-[x_{24},i+1,i-1],\\
\partial [x_2,i+1,i+1]&=[x_{1},i+1,i]-[x_{3},i,i+1]-[x_{9},i,i+1]-[x_{23},i+1,i],\\
\partial [x_3,i,i+1]&=[x_{7},i-1,i]-[x_{8},i-1,i+1]+[x_{21},i,i-1]-[x_{22},i,i],\\
\partial [x_4,i,i]&=-[x_{7},i-1,i]-[x_{21},i,i-1],\\
\partial [x_5,i,i+1]&=[x_{4},i,i]-[x_{6},i-1,i+1]-[x_{20},i,i],\\
\partial [x_6,i-1,i+1]&=-[x_{7},i-1,i]-[x_{19},i-1,i],\\
\partial [x_7,i-1,i]&=-[x_{18},i-1,i-1],\\
\partial [x_8,i-1,i+1]&=-[x_{17},i-1,i],\\
\partial [x_9,i,i+1]&=[x_{8},i-1,i+1]-[x_{10},i,i]-[x_{16},i,i],\\
\partial [x_{10},i,i]&=-[x_{15},i,i-1],\\
\partial [x_{11},i-1,i]&=-[x_{14},i-1,i-1],\\
\partial [x_{12},i,i]&=[x_{11},i-1,i]-[x_{13},i,i-1],\\
\partial [x_{13},i,i-1]&=-[x_{14},i-1,i-1],\\
\partial [x_{14},i-1,i-1]&=0,\\
\partial [x_{15},i,i-1]&=0,\\
\partial [x_{16},i,i]&=[x_{15},i,i-1]-[x_{17},i-1,i],\\
\partial [x_{17},i-1,i]&=0,\\
\partial [x_{18},i-1,i-1]&=0,\\
\partial [x_{19},i-1,i]&=[x_{18},i-1,i-1],\\
\partial [x_{20},i,i]&=[x_{19},i-1,i]-[x_{21},i,i-1],\\
\partial [x_{21},i,i-1]&=[x_{18},i-1,i-1],\\
\partial [x_{22},i,i]&=[x_{17},i-1,i],\\
\partial [x_{23},i+1,i]&=[x_{16},i,i]+[x_{22},i,i]-[x_{24},i+1,i-1],\\
\partial [x_{24},i+1,i-1]&=[x_{15},i,i-1],\\
\partial [x_{25},i,i-1]&=[x_{14},i-1,i-1],\\
\partial [x_{26},i+1,i-1]&=[x_{13},i,i-1]+[x_{25},i,i-1],\\
\partial [x_{27},i+1,i]&=[x_{12},i,i]+[x_{26},i+1,i-1]-[x_{28},i,i],\\
\partial [x_{28},i,i]&=[x_{11},i-1,i]+[x_{25},i,i-1],\\
\partial [x_{29},i+1,i]&=[x_{10},i,i]-[x_{11},i-1,i]+[x_{24},i+1,i-1]-[x_{25},i,i-1],\\
\partial [x_{30},i+1,i+1]&=[x_{9},i,i+1]+[x_{23},i+1,i]+[x_{29},i+1,i]-[x_{31},i,i+1],\\
\partial [x_{31},i,i+1]&=[x_{8},i-1,i+1]-[x_{11},i-1,i]+[x_{22},i,i]-[x_{25},i,i-1],
\end{align*}
by counting all Whitney disks of unit Maslov index, where the signs are by an orientation of the $\beta$ curve.

By definition, the chain complex $\widehat{\text{CF}}(S^3)$ is generated by $[x_i,0,a_i]$ ($i=1,\ldots, 64n+31$), and the boundary operator on $\widehat{\text{CF}}(S^3)$ is given by

\begin{align*}
\partial [x_1,0,-1]&=-[x_{24},0,-2],\\
\partial [x_2,0,0]&=[x_{1},0,-1]-[x_{23},0,-1],\\
\partial [x_3,0,1]&=[x_{21},0,-1]-[x_{22},0,0],\\
\partial [x_4,0,0]&=-[x_{21},0,-1],\\
\partial [x_5,0,1]&=[x_{4},0,0]-[x_{20},0,0],\\
\partial [x_6,0,2]&=-[x_{7},0,1]-[x_{19},0,1],\\
\partial [x_7,0,1]&=-[x_{18},0,0],\\
\partial [x_8,0,2]&=-[x_{17},0,1],\\
\partial [x_9,0,1]&=-[x_{10},0,0]-[x_{16},0,0],\\
\partial [x_{10},0,0]&=-[x_{15},0,-1],\\
\partial [x_{11},0,1]&=-[x_{14},0,0],\\
\partial [x_{12},0,0]&=-[x_{13},0,-1],\\
\partial [x_{13},0,-1]&=0,\\
\partial [x_{14},0,0]&=0,\\
\partial [x_{15},0,-1]&=0,\\
\partial [x_{16},0,0]&=[x_{15},0,-1],\\
\partial [x_{17},0,1]&=0,\\
\partial [x_{18},0,0]&=0,\\
\partial [x_{19},0,1]&=[x_{18},0,0],\\
\partial [x_{20},0,0]&=-[x_{21},0,-1],\\
\partial [x_{21},0,-1]&=0,\\
\partial [x_{22},0,0]&=0,\\
\partial [x_{23},0,-1]&=-[x_{24},0,-2],\\
\partial [x_{24},0,-2]&=0,\\
\partial [x_{25},0,-1]&=0,\\
\partial [x_{26},0,-2]&=0,\\
\partial [x_{27},0,-1]&=[x_{26},0,-2],\\
\partial [x_{28},0,0]&=[x_{25},0,-1],\\
\partial [x_{29},0,-1]&=[x_{24},0,-2],\\
\partial [x_{30},0,0]&=[x_{23},0,-1]+[x_{29},0,-1],\\
\partial [x_{31},0,1]&=[x_{22},0,0]-[x_{25},0,-1].
\end{align*}
The homology of is generated by the cycle $[x_3,0,1]+[x_4,0,0]+[x_{28},0,0]+[x_{31},0,1]$. Since $[x_3,0,1]+[x_4,0,0]+[x_{28},0,0]+[x_{31},0,1]\in \mathcal{F}(K_0,1)$, we have $\tau(K_0)\le 1$.

Consider the projection $p:\widehat{\text{CF}}(S^3)\to\mathbf{Z}$ that sends all generators to zero except for $p([x_3,0,1])=1$. Then $p$ is a chain map which induces isomorphism on homology. Since $p$ maps all elements in $\mathcal{F}(K_0,0)$ to zero, we have $\tau(K_0)\ge 1$.

Therefore the invariant of $\tau(K_0)$ is $1$.
\end{proof}
\begin{theorem}
The $\tau$ invariant and the smooth slice genus of $K_n$ are $1$.
\end{theorem}
\begin{proof}
After changing $n$ of the $2n$ positive crossings in white circle in Figure \ref{fig:k_n} to negative crossings, the new knot becomes $K_0$, therefore by \cite{ozs} we have $g_4(K_n)\ge \tau(K_n)\ge \tau(K_0)=1$.

By resolving one of the $2n$ positive crossings in white circle and the crossing $a$ in Figure \ref{fig:k_n}, we get a new knot as follows.

\begin{figure}[H]
\centering
\includegraphics[width=0.4\textwidth]{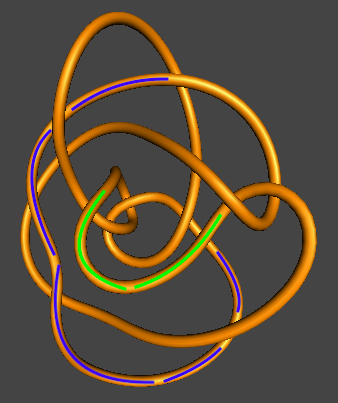}
\caption{\label{fig:k_new}The knot obtained by resolving two crossings in $K_n$.}
\end{figure}

By straightening the blue part and the green part in Figure \ref{fig:k_new}, we get the following isotopic knot.

\begin{figure}[H]
\centering
\includegraphics[width=0.4\textwidth]{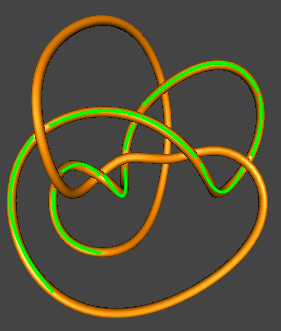}
\caption{\label{fig:k_new2}The knot obtained by resolving two crossings in $K_n$.}
\end{figure}

By straightening the green part in Figure \ref{fig:k_new2}, it is now clear that our new knot is the unknot.

Hence, by resolving two crossings, we can construct a torus cobordism (a split cobordism followed by a merge cobordism) between $K_n$ and the unknot. By definition, we have $g_4(K_n)\le 1$. 

Therefore, we have $\tau(K_n)=g_4(K_n)=1$.
\end{proof}
\begin{theorem}
There are infinitely many $(1,1)$-knots which are topologically slice, but not smoothly slice.
\end{theorem}
\begin{proof}
The family of knots $K_n$ serves as an example.
\end{proof}

\bibliographystyle{alpha}

\end{document}